\documentclass[reqno]{amsart}
\usepackage[latin1]{inputenc}
\usepackage[T1]{fontenc}
\usepackage[english]{babel}
\usepackage{amssymb,amsmath}
\usepackage{amsthm}
\usepackage{mathrsfs}
\usepackage{geometry}

\newtheorem{thm}{Theorem}[section]
\newtheorem{defini}{Definition}[section]

\newtheorem{prop}{Proposition}[section]
\newtheorem{coro}{Corollary}[section]

\begin{document}
	\title[ ]{ Semi-Fredholm and Semi-Browder Spectra For $C_0$-quasi-semigroups  }
	
	\author[ Abdelaziz Tajmouati, Youness Zahouan 
	\\]
	{ A. Tajmouati \; , \; Y. Zahouan}
	\address{A. Tajmouati, Y. Zahouan \newline
		Sidi Mohamed Ben Abdellah
		Univeristy,
		Faculty of Sciences Dhar Al Mahraz, Fez,  Morocco.}
	\email{abdelaziz.tajmouati@usmba.ac.ma}
	\email{zahouanyouness1@gmail.com}
	\maketitle

\begin{center}
	\textbf{Abstract}
\end{center}
		In this paper, we describe the different spectra of the $C_0$-quasi-semigroups by the spectra of their generators. Specially, essential ascent and descent,Drazin invertible, upper and lower semi-Fredholm and semi-Browder spectra.

\textbf{keywords}:
$C_0$-quasi-semigroup,$C_0$-semigroup ,semi-Fredholm , ascent ,  descent spectrum ,  Drazin spectrum , ,semi-Browder spectrum.

	\section{\textbf{Introduction}}
	Let $X$ be a complex Banach space and $\mathcal{B}(X)$ the algebra of all bounded linear operators on  $X$. We denote by  $Rg(T)$, $N(T)$, $\rho(T)$ and $\sigma(T),$  respectively  the range, the kernel, the resolvent and the spectrum of $T$, where
	$\sigma(T)=\{\lambda\in\mathbb{C} \,/ \, \lambda-T \,\mbox{is not bijective}\}$.
	 The function resolvent of $T\in\mathcal{B}(X)$ is defined for all $\lambda\in\rho(T)$ by $\mathcal{R}(\lambda,T)=(\lambda -T)^{-1}.$ The ascent and descent of an operator $T$ are defined respectively by,
	 \begin{center}
	 		$a(T)=\min\{k\in\mathbb{N}\,/ \,N(T^k)=N(T^{k+1})\} \; ; \;  d(T)=\min\{k\in\mathbb{N}\,/\, Rg(T^k)=Rg(T^{k+1})\}$.
	 \end{center}
	   with the convention $inf(\varnothing)=\infty$.\\
	The essential ascent and descent of an operator $T$ are defined respectively by, \\
	$a_e(T)=\min\{k\in\mathbb{N}\,/ \,dim[N(T^{k+1})/N(T^k)]< \infty\}$ ; 	$d_e(T)=\min\{k\in\mathbb{N} / \, dim[Rg(T^k)/Rg(T^{k+1})]<\infty\}$.\\
	The ascent, descent , essential ascent and essential descent spectra are defined by,
	
	\begin{itemize}
		\item $\sigma_a(T)=\{\lambda\in\mathbb{C}\,/ \, a(\lambda-T)=\infty\}$ \; \;  ; \; 
		$\sigma_d(T)=\{\lambda\in\mathbb{C}\,/ \, d(\lambda-T)=\infty\}$.
		\item $\sigma_{a_e}(T)=\{\lambda\in\mathbb{C}\,/ \, a_e(\lambda-T)=\infty\}$ \; ; \;  $\sigma_{d_e}(T)=\{\lambda\in\mathbb{C}\,/\, d_e(\lambda-T)=\infty\}$.
	\end{itemize}
	
	We say that a closed linear operator $A$ is Drazin invertible if $  d(T) = a(T) = p < \infty$  and $Rg(A^{p})$ is closed. The Drazin spectrum is $$\sigma_D(T)=\{\lambda\in\mathbb{C}\,/ \, \lambda-T \; \mbox{si not  Drazin invertible} \}$$

	The sets of upper and lower semi-Fredholm and their spectra are defined respectively by,
	\begin{itemize}
	\item $\Phi_+(X)=\{T\in\mathcal{B}(X)\,/ \, \alpha(T):=dim(N(T))<\infty \;\mbox{and}\;Rg(T) \,\mbox{is closed}\}$\\
	$\sigma_{e_+}(T)=\{\lambda\in\mathbb{C}\,/ \,\lambda-T\notin \Phi_+(X)\}$ \\
	
	\item $\Phi_-(X)=\{T\in\mathcal{B}(X)\,/ \, \beta(T):=codim(Rg(T))=dim(X \backslash Rg(T))<+\infty\}$\\
$\sigma_{e_-}(T)=\{\lambda\in\mathbb{C} / \lambda-T\notin \Phi_-(X)\}.$
	\end{itemize}
	An operator $T\in \mathcal{B}(X)$ is called semi-Fredholm, in symbol $T\in \Phi_\pm(X)$, if $T\in \Phi_+(X)\cup\Phi_-(X).$\\
	An operator $T\in \mathcal{B}(X)$ is called Fredholm, in symbol $T\in \Phi(X)$, if $T\in \Phi_+(X)\cap\Phi_-(X)$.\\
	The essential and semi-Fredholm spectra are defined by,
	
	\begin{center}
		\begin{itemize}
		\item 	$\sigma_e(T)=\{\lambda\in\mathbb{C} / \lambda-T\notin \Phi(X)\}$
		\item $\sigma_{e_\pm}(T)=\{\lambda\in\mathbb{C} / \lambda-T\notin \Phi_\pm(X)\}$
			
		\end{itemize}
	\end{center}
	
	The sets of upper and lower semi-Browder and their spectra are defined respectively by,
		\begin{itemize}
		\item ${Br}_+(X)=\{T\in\Phi_+(X)\,/ \, a(T)<+\infty\}$ ; 
		$\sigma_{{Br}_+}(T)=\{\lambda\in\mathbb{C} / \lambda-T\notin Br_+(X)\}$
		\item ${Br}_-(X)=\{T\in\Phi_-(X)\,/ \, d(T)<+\infty\}$ ; 
		$\sigma_{Br_-}(T)=\{\lambda\in\mathbb{C} / \lambda-T\notin Br_-(X)\}$
		
	\end{itemize}
	
	An operator $T\in \mathcal{B}(X)$ is called semi-Browder, in symbol $T\in Br_\pm(X)$, if	$T\in {Br}_+(X)\cup {Br}_-(X)$.\\
	An operator $T\in \mathcal{B}(X)$ is called Browder, in symbol $T\in Br(X)$, if
	$T\in Br_+(X)\cap Br_-(X)$.\\
	The semi-Browder and Browder spectra are defined by,
	
	\begin{itemize}
		\item $\sigma_{{Br}_{\pm}}(T)=\{\lambda\in\mathbb{C} / \lambda-T\notin {Br}_\pm(X)\}$.
		\item $\sigma_{Br}(T)=\{\lambda\in\mathbb{C} /  \lambda-T\notin Br(X)\}$.
		
	\end{itemize}

	The theory of quasi-semigroups of  bounded linear operators, as a generalization of semigroups of operators, was introduced by Leiva and Barcenas \cite{r.3} , \cite{r.4} ,\cite{r.5}.Recently Sutrima , Ch. Rini Indrati and others \cite{r.9} , there are show some relations between a $C_{0}$-quasi-semigroup and its  generator  related  to  the  time-dependent evolution equation.
	
	\begin{defini}\cite{r.3}
	Let $X$ be a complex Banach. The family $\left\lbrace R(t,s)\right\rbrace _{t,s\geq 0} \subseteq \mathcal{B}(X)$   is called  a  strongly  continuous  quasi-semigroup   (or   $C_{0}$-quasi-semigroup)  of  operators  if  for  every  $t,s,r\geq 0$  and   $x \in X$ , 
	\begin{enumerate}
		\item $R(t,0)= I $, the identity operator on  $X$,
		\item $R(t,s+r)=R(t+r,s)R(t,r)$,
		\item $lim_{s\longrightarrow 0}\left| \left|R(t,s)x - x \right| \right| = 0$ , 
		\item There exists a continuous increasing mapping $ M : [ 0 ; +\infty [ \longrightarrow [ 1 ; +\infty [ $ such that ,
		\begin{center}
			$\left| \left|R(t,s)\right| \right| \leq  M(t+s)$
		\end{center}
		
	\end{enumerate}

	\end{defini}
	
	\begin{defini}\cite{r.3}
		For a  $C_{0}$-quasi-semigroup  $\left\lbrace R(t,s)\right\rbrace _{t,s\geq 0}$ on a Banach space   $X$, let $\mathcal{D}$ be the set of all   $x \in X$    for which 
		the following limits exist,
		\begin{center}
			$ \lim_{s\rightarrow 0^+}\frac{R(0,s)x-x}{s}$ \; \; and  \; \; $ \lim_{s\rightarrow 0^+}\frac{R(t,s)x-x}{s} =  \lim_{s\rightarrow 0^+}\frac{R(t-s,s)x-x}{s}$ , \; $ t> 0 $
		\end{center}
		
			\begin{itemize}
		\item For  $t\geq 0$    we define an operator $A(t)$  on $\mathcal{D}$ as  $A(t)x = \lim_{s\rightarrow 0^+}\frac{R(t,s)x-x}{s}$
		
		The family  $\left\lbrace A(t)\right\rbrace _{t\geq 0}$ is called infinitesimal generator of the $C_{0}$-quasi-semigroups $\left\lbrace R(t,s)\right\rbrace _{t,s\geq 0}$.
		
     	\item For  each   $t\geq 0$  we  define  the resolvent operator of  $A(t)$  as  
	$ \mathcal{R}(\lambda,A(t))=  (\lambda I - A(t))^{-1}$, with its resolvent set $\rho(A(t))$.
		\end{itemize}
\end{defini}

		Throughout this paper we  denote 
		$T(t)$  and  $R(t,s)$ as $C_{0}$-semigroups  $\left\lbrace T(t)\right\rbrace _{t\geq 0}$ and  $C_{0}$-quasi-semigroup  $\left\lbrace R(t,s)\right\rbrace _{t,s\geq 0}$  respectively.We also denote $\mathcal{D}$ as domain for  $A(t)$ , $t\geq 0$.\\

			\begin{thm}\label{l0}\cite{r.9}
			
			Let $R(t,s)$ is a $C_0$-quasi-semigroup on  $X$ with  generator  $A(t)$ then,
			
			\begin{enumerate}
				\item For each   $t\geq 0$  ,  $R(t,.)$ is strongly continuous on  $ [ 0 ; +\infty [ $.
				\item For each   $t\geq 0$ and  $x\in X$,
				\begin{center}
					$	\lim_{s\rightarrow 0^+} \frac{1}{s}\int_{0}^{s}R(t,h) x dh = x$
				\end{center}
				\item If $x \in \mathcal{D} $ ,  $t\geq 0$ and $t_{0},s_{0}\geq 0$ then,  $R(t_{0},s_{0})x \in \mathcal{D} $ and 
				\begin{center}
					$R(t_{0},s_{0})A(t)x = A(t)R(t_{0},s_{0})x $
				\end{center}
				\item For each $s > 0$ , $\frac{\partial}{\partial s} (R(t,s)x)= A(t+s)R(t,s)x = R(t,s) A(t+s)x $;  $ x \in \mathcal{D}. $
				\item If  $A(.)$ is locally integrable, then for every $x \in \mathcal{D} $ and   $s \geq 0$,
				
				\begin{center}
					$	R(t,s)x = x + \int_{0}^{s} A(t+h)R(t,h)x dh .  $
				\end{center}
				\item If $ f : [ 0 ; +\infty [ \longrightarrow X$  is a continuous, then for every $t \in [ 0 ; +\infty [$  
				
				\begin{center}
					$	\lim_{r\rightarrow 0^+} \int_{s}^{s+r}R(t,h) f(h)x dh = R(t,s)f(s)$
				\end{center}
			\end{enumerate}
				
		\end{thm}
	In the semigroups theory, if $A$ is an infinitesimal generator of $C_0$-semigroup with domain $D(A)$  , then $A$  is a closed operator  and $D(A)$  is dense in $X$.  These are is not always true for any $C_0$-quasi-semigroups, see \cite{r.9}.\\

	In \cite{r.6},\cite{r.7},\cite{r.8},\cite{r.10},\cite{r.11}, the authors have studied the different spectra of the $C_0$-semigroups.
	In our paper \cite{r.12}, we have studied  ordinary , point , approximate point , residual , essential and regular spectra of the $C_0$-quasi-semigroups. In this paper, we continue to study $C_0$-quasi-semigroups. We investigate the relationships between the different spectra of the $C_0$-quasi-semigroups and their generators, precisely the essential ascent and descent, Drazin, upper and lower semi-Fredholm and semi-Browder spectra.
	
	
	\section{\textbf{Main results}}

	In \cite{r.12}  we have proved the following theorem,
	\begin{thm}\label{t1} 
	 Let $A(t) $ be the generator of the C$_{0}-$quasi-semigroup  $(R(t,s))_{t,s\geq 0}$ . Then for all $\lambda\in\mathbb{C}$ and all $t,s\geq 0$ we have
			\begin{enumerate}
				\item For all $x\in X$ we have $$(\lambda - A(t))D_\lambda(t,s)x=[e^{\lambda s} -R(t,s)]x .$$
				\item For all $x\in \mathcal{D}$ we have $$D_\lambda(t,s)(\lambda - A(t))x=[e^{\lambda s} - R(t,s)]x .$$
			\end{enumerate}
			With $D_\lambda(t,s)x = \int_0^s e^{\lambda(s-h)}R(t,h)xdh$ for all $x\in X$  and $t,s\geq 0$ is a bounded linear operator on $X$.
		
	\end{thm}
	
	\begin{proof}
		\begin{enumerate}
			\item For all $x\in X$ we have
			
			\begin{eqnarray*}
				R(0,r)D_\lambda(t,s)x   &=&   R(0,r) \int_0^s e^{\lambda(s-h)}R(t,h)xdh\\
				&=&\int_0^s e^{\lambda(s-h)}R(0,r)R(t,h)xdh \\
			\end{eqnarray*}

			And  we obtain ,
			
			\begin{eqnarray*}
				\lim_{r\rightarrow 0^+}\frac{R(0,r)D_\lambda(t,s)x - D_\lambda(t,s)x}{r} &=& \lim_{r\rightarrow 0^+}\frac{\int_0^s e^{\lambda(s-h)}R(0,r)R(t,h)xdh -\int_0^s e^{\lambda(s-h)}R(t,h)xdh}{r}\\
				&=&\frac{\partial}{\partial r}\left[ \int_0^{s}e^{\lambda(s-h)}R(0,r)R(t,h)xdh \right] _{|r=0} \\
			\end{eqnarray*}

			Then  $\lim_{r\rightarrow 0^+}\frac{R(0,r)D_\lambda(t,s)x - D_\lambda(t,s)x}{r} $   exists.\\

			And,
			
			\begin{eqnarray*}
				\lim_{r\rightarrow 0^+}\frac{R(t,r)D_\lambda(t,s)x - D_\lambda(t,s)x}{r}  &=&   \lim_{r\rightarrow 0^+}\frac{\int_0^s e^{\lambda(s-h)}R(t,r)R(t,h)xdh -\int_0^s e^{\lambda(s-h)}R(t,h)xdh}{r}\\
				&=& \frac{\partial}{\partial r}\left[ \int_0^{s}e^{\lambda(s-h)}R(t,r)R(t,h)xdh \right] _{|r=0} \\
				&=&\left[ \int_0^{s}e^{\lambda(s-h)}\frac{\partial}{\partial r}(R(t,r))R(t,h)xdh \right] _{|r=0} \\
				&=& \left[ \int_0^{s}e^{\lambda(s-h)}A(t+r)R(t,r)R(t,h)xdh \right] _{|r=0} \\
				&=& \int_0^{s}e^{\lambda(s-h)}A(t)R(t,h)xdh      
			\end{eqnarray*}

			Moreover,
			\begin{eqnarray*} 
				\lim_{r\rightarrow 0^+}\frac{R(t - r,r)D_\lambda(t,s)x - D_\lambda(t,s)x}{r}  &=&    \lim_{r\rightarrow 0^+}\frac{\int_0^s e^{\lambda(s-h)}R(t-r,r)R(t,h)xdh -\int_0^s e^{\lambda(s-h)}R(t,h)xdh}{r}\\
				&=& \frac{\partial}{\partial r}\left[ \int_0^{s}e^{\lambda(s-h)}R(t-r,r)R(t,h)xdh \right] _{|r=0} \\
				&=&\left[ \int_0^{s}e^{\lambda(s-h)}\frac{\partial}{\partial r}(R(t-r,r))R(t,h)xdh \right] _{|r=0} \\
				&=&\left[ \int_0^{s}e^{\lambda(s-h)}A(t)R(t-r,r)R(t,h)xdh \right] _{|r=0} \\
				&=&\int_0^{s}e^{\lambda(s-h)}A(t)R(t,h)xdh     
			\end{eqnarray*}

			Thus , $\lim_{r\rightarrow 0^+}\frac{R(t,r)D_\lambda(t,s)x - D_\lambda(t,s)x}{r}  = \lim_{r\rightarrow 0^+}\frac{R(t - r,r)D_\lambda(t,s)x - D_\lambda(t,s)x}{r}$

			Hence, we deduce that $D_\lambda(t , s)x\in \mathcal{D}$ And , \\
			
			\begin{eqnarray*}
				A(t)D_\lambda(t,s)x  &=& \int_0^{s}e^{\lambda(s-h)}A(t)R(t,h)xdh \\
				&=&  \int_0^{s}e^{\lambda(s-h)}A(t+h)R(t,h)xdh ,\; \; because \; \;  A(t+h) = A(t) for \; all \; \;   t,h \geq 0 .\\
				&=&  \int_0^{s}e^{\lambda(s-h)}\frac{\partial}{\partial h}(R(t,h))xdh \\
				&=&  \left[ e^{\lambda(s-h)}R(t,h) \right]_{0}^{s} + \lambda\int_0^{s}e^{\lambda(s-h)}R(t,h)xdh \\
				&=&  R(t,s)x - e^{\lambda s}x + \lambda D_\lambda(t,s)x
			\end{eqnarray*}                    
			
			Finaly , $(\lambda - A(t))D_\lambda(t,s)x=[e^{\lambda s} -R(t,s)]x $ for all $x \in X $.

			\item For all $x\in \mathcal{D} $ and all $t ,s \geq 0$ we have ,
			\begin{eqnarray*}
				D_\lambda(t ,s)A(t)x &=& \int_0^s e^{\lambda(s-h)}R(t,h)A(t)xdh\\
				&=& \int_0^s e^{\lambda(s-h)}R(t,h)A(t+h)xdh\\
				&=& \int_0^{s}e^{\lambda(s-h)}\frac{\partial}{\partial r}(R(t,h))xdh\\
				&=&  \left[ e^{\lambda(s-h)}R(t,h) \right]_{0}^{s} + \lambda\int_0^{s}e^{\lambda(s-h)}R(t,h)xdh \\
				&=&  R(t,s)x - e^{\lambda s}x + \lambda D_\lambda(t,s)x
			\end{eqnarray*}
			Thus, we deduce for all $x\in D(A)$
			$D_\lambda(t,s)(\lambda - A(t))x=[e^{\lambda s} - R(t,s)]x $.
		\end{enumerate}
	\end{proof}
	

		\begin{coro}\label{c1}  Let $A(t)$ be the generator of a $C_{0}$-quasi-semigroup $(R(t,s))_{t,s\geq 0}$ . Then for all $\lambda \in \mathbb{C}$ , $t,s\geq 0$ and $n \in \mathbb{N}$ ,
		\begin{enumerate}
			\item For all  $x\in X$ ,
			\begin{center}
				$(\lambda-A(t))^n[D_\lambda(t,s)]^nx=[e^{\lambda s}-R(t,s)]^nx$.
			\end{center}
			\item For all  $x\in \mathcal{D}^n$ \, (Domain of $A(t)^n$)  ,  
			\begin{center}
				$[D_\lambda(t,s)]^n(\lambda-A(t)^n)x=[e^{\lambda s}-R(t,s)]^nx.$
			\end{center}
			\item $N[\lambda-A(t)]\subseteq N[e^{\lambda s}-R(t,s)].$
			\item $Rg[e^{\lambda s}-R(t,s)]\subseteq Rg[\lambda-A(t)].$
			\item $N[\lambda-A(t)]^n\subseteq N[e^{\lambda s}-R(t,s)]^n.$
			\item $Rg[e^{\lambda s}-R(t,s)]^n\subseteq Rg[\lambda-A(t)]^n.$\\
		\end{enumerate}
	\end{coro}
	\begin{proof}
		
		follow easily from ,
		\begin{eqnarray*}
			e^{\lambda s}x - R(t,s)x &=& (\lambda - A(t))D_\lambda(t,s)x \\
			&=&   D_\lambda(t,s)(\lambda - A(t))x
		\end{eqnarray*}
		\end {proof}


To obtain the results concerning the semi-Fredholm and semi-Browder spectra we need the following theorem.
		\begin{thm}\label{t2}Let $A(t) $ be a closed  and densely defined  generator of  a $C_{0}-$quasi-semigroup  $(R(t,s))_{t,s\geq 0}$ on  a 	Banach  space $X$ . Then for all $\lambda\in\mathbb{C}$ and all $t,s\geq 0$ we have,
			\begin{enumerate}
				\item $(\lambda-A(t))L_\lambda(t,s)+\varphi_\lambda(s)D_\lambda(t,s)= \phi_\lambda(s)I ,$
				where $L_\lambda(t,s)=\int_0^s e^{-\lambda h }D_\lambda(t,h)dh\,\,$ , $\,\,\varphi_\lambda(s)=e^{-\lambda s}$ and $\phi_\lambda(s) =s$.
				
				Moreover, the operators $L_\lambda(t,s),$ $D_\lambda(t,s)$ and $(\lambda-A(t))$ are mutually commuting.
				
				\item For all $n\in\mathbb{N}^*,$ there exists an operator $D_{\lambda,n}(t,s)\in \mathcal{B}(X)$ such that, 
				\begin{center}
					$(\lambda-A(t))^n[L_\lambda(t,s)]^n+D_{\lambda,n}(t,s)D_\lambda(t,s)=s^{n}.$
				\end{center}
				Moreover, the operator $D_{\lambda,n}(t,s)$ is commute with each one of $D_\lambda(t,s)\,$ and $\,L_\lambda(t,s)$.
				
				\item  For all $n\in\mathbb{N}^*,$ there exists an operator $K_{\lambda,n}(t,s)\in \mathcal{B}(X)$ such that,
					$$(\lambda-A(t))^nK_{\lambda,n}(t,s)+[D_{\lambda,n}(t,s)]^n[D_\lambda(t,s)]^n =s^{n^2}.$$
	
				Moreover, the operator $K_{\lambda,n}(t,s)$ is commute with each one of $D_\lambda(t,s)\,$ and $\,D_{\lambda,n}(t,s)$.
			\end{enumerate}
		\end{thm}
		\begin{proof}
			\begin{enumerate}
				\item Let $\mu\in \rho(A(t))$. By theorem  \ref{t2}, for all $x\in X$ we have $D_\lambda(t,h)x\in \mathcal{D}$ and hence,for all $t,s \geq 0$,
				\begin{eqnarray*}
					L_\lambda(t,s)x &=& \int_0^se^{-\lambda h}D_\lambda(t,h)xdh\\
					&=& \int_0^se^{-\lambda h}\mathcal{R}(\mu,A(t))(\mu-A(t))D_\lambda(t,h)xdh , \; \; where \; \mathcal{R}(\mu,A(t)) \; is \; the \;  resolvent \;  of \; A(t)\\
					&=& \mathcal{R}(\mu,A(t))[\mu\int_0^se^{-\lambda h}D_\lambda(t,h)xdh-\int_0^se^{-\lambda h}A(t)D_\lambda(t,h)xdh]\; \; \; (\mathcal{R}  (\mu,A(t)) \; is \; bounded)\\
					&=& \mathcal{R}(\mu,A(t))[\mu L_\lambda(t,s)x-\int_0^se^{-\lambda h}A(t)D_\lambda(t,h)xdh]
				\end{eqnarray*}
				Therefore for all $x\in X$, we have $L_\lambda(t,s)x\in \mathcal{D}$ and
				$$(\mu-A(t))L_\lambda(t,s)x=\mu L_\lambda(t,s)x-\int_0^se^{-\lambda h}A(t)D_\lambda(t,h)xdh.$$
				Thus
				$$A(t)L_\lambda(t,s)x=\int_0^se^{-\lambda h}A(t)D_\lambda(t,h)xdh.$$
				Hence, we conclude that
				\begin{eqnarray*}
					(\lambda-A(t))L_\lambda(t,s)x&=&\lambda L_\lambda(t,s)x -\int_0^se^{-\lambda h}A(t)D_\lambda(t,h)xdh\\
					&=&\lambda L_\lambda(t,s)x -\int_0^se^{-\lambda h}\big[\lambda D_\lambda(t,h)x- e^{\lambda h}x +
					R(t,h)x\big]dh \; \; (Theorem \, \ref{t2})\\
					&=&\lambda L_\lambda(t,s)x -\lambda\int_0^se^{-\lambda h}D_\lambda(t,h)x ds +\int_0^s x -\int_0^se^{-\lambda sh}R(t,h)xdh\\
					&=&\lambda L_\lambda(t,s)x -\lambda L_\lambda(t,s)x + sx -e^{-\lambda s}\int_0^se^{\lambda(s-h)}R(t,h)xdh  \\
					&=& sx -e^{-\lambda s}D_\lambda (t,s)x  \\
					&=& \big[\phi_\lambda(s)-\varphi_\lambda(s)D_\lambda(t,s)\big]x,
				\end{eqnarray*}
				where $\phi_\lambda(s)=s$ and $\varphi_\lambda(s)=e^{-\lambda s}.$\\
				Therefore, we obtain $(\lambda-A(t))L_\lambda(t,s)+\varphi_\lambda(s)D_\lambda(t,s)=\phi_\lambda(s)I.$ \\
				On the other hand,for all $s > t <r\geq 0$ we have ,$R(t,s)R(t,r)=R(t-s,s)R(t,r)=R(t-s-r,s+r)$ and 
				$R(t,r)R(t,s)=R(t-r,r)R(t,s)=R(t-r-s-r,s+r)$
				Then $R(t,s)R(t,r)=R(t,r)R(t,s)$ for all $s,t,r\geq 0$, then
				$D_\lambda(t,h)R(t,s)=R(t,s)D_\lambda(t,h).$\\
				Hence
				\begin{eqnarray*}
					D_\lambda(t,s)D_\lambda(t,r) &=& \int_0^se^{\lambda(s-h)}R(t,h)D_\lambda(t,r)dh\\
					&=&  \int_0^se^{\lambda(s-h)}D_\lambda(t,r)R(t,h)dh\\
					&=& D_\lambda(t,r)\int_0^se^{\lambda(s-h)}R(t,h)dh\\
					&=& D_\lambda(t,r)D_\lambda(t,s).
				\end{eqnarray*}
				Thus, we deduce that
				\begin{eqnarray*}
					D_\lambda(t,s)L_\lambda(t,s) &=& D_\lambda(t,s)\int_0^se^{-\lambda h}D_\lambda(t,h)dh\\
					&=& \int_0^te^{-\lambda s}D_\lambda(t,s)D_\lambda(t,h)dh\\
					&=& \int_0^se^{-\lambda h}D_\lambda(t,h)D_\lambda(t,s)dh\\
					&=& \int_0^se^{-\lambda h}D_\lambda(t,h)dh D_\lambda(t,s)\\
					&=& L_\lambda(t,s)D_\lambda(t,s).
				\end{eqnarray*}
				Since for all $x\in X$, $A(t)L_\lambda(t,s)x=\int_0^se^{-\lambda h}A(t)D_\lambda(t,h)xdh$ and for all $x\in \mathcal{D}$, $A(t)D_\lambda(t,h)x=D_\lambda(t,h)A(t)x,$
				then we obtain for all $x\in \mathcal{D}$,
				\begin{eqnarray*}
					(\lambda-A(t))L_\lambda(t,s)x &=& \lambda L_\lambda(t,s)x-A(t)L_\lambda(t)x\\	     		                   &=& \lambda L_\lambda(t,s)x -\int_0^se^{-\lambda h}A(t)D_\lambda(t,h)xdh\\
					&=&\lambda L_\lambda(t,s)x -\int_0^se^{-\lambda h}D_\lambda(t,h)A(t)xdh\\
					&=& \lambda L_\lambda(t,s)x- L_\lambda(t,s)A(t)x\\
					&=& L_\lambda(t,s)(\lambda-A(t))x.
				\end{eqnarray*}

				\item For all $n\in \mathbb{N}^*$, we obtain
				\begin{eqnarray*}
					[(\lambda-A(t))L_\lambda(t,s)]^n &=&[s-\varphi_\lambda(s)D_\lambda(t,s)]^n\\
					&=&\sum_{i=0}^n C_n^is^{n-i}[-\varphi_\lambda(s)D_\lambda(t,s)]^i\\
					&=& s^n + \sum_{i=1}^n C_n^is^{n-i}[-\varphi_\lambda(s)D_\lambda(t,s)]^i\\
					&=& s^n -D_\lambda(t,s)\sum_{i=1}^nC_n^i s^{n-i}[\varphi_\lambda(s)]^{i}[-D_\lambda(t,s)]^{i-1}\\
					&=& s^n -D_\lambda(t,s)D_{\lambda,n}(t,s),
				\end{eqnarray*}
				where $$D_{\lambda,n}(t,s)=\sum_{i=1}^n C_n^i s^{n-i}[\varphi_\lambda(s)]^{i}[-D_\lambda(t,s)]^{i-1}.$$
				Therefore, we have
				$$(\lambda-A(t))^n[L_\lambda(t,s)]^n+D_\lambda(t,s)D_{\lambda,n}(t,s)=s^n.$$
				Finally, for commutativity, it is clear that $D_{\lambda,n}(t,s)$ commute with each one of $D_\lambda(t,s)\,$ and $\,L_\lambda(t,s)$ since the operators $L_\lambda(t,s),$ $D_\lambda(t,s)$ and $(\lambda-A(t))$ are mutually commuting from (1).
				\item
				Since we have
				$D_\lambda(t,s)D_{\lambda,n}(t,s)=s^n-(\lambda-A(t))^n[L_\lambda(t,s)]^n,$
				then for all $n\in \mathbb{N}$
				\begin{eqnarray*}
					[D_\lambda(t,s)D_{\lambda,n}(t,s)]^n &=& \big[s^n-(\lambda-A(t))^n[L_\lambda(t,s)]^n\big]^n\\
					&=& s^{n^2}-\sum_{i=1}^n C_n^i \big[s^{n}\big]^{n-i}\big[(\lambda-A(t))^n[L_\lambda(t,s)]^n\big]^i\\
					&=& s^{n^2}-(\lambda-A(t))^n\sum_{i=1}^nC_n^i \big[s^{n(n-i)}(\lambda-A(t))^{n(i-1)}[L_\lambda(t,s)]^{ni}\\
					&=& s^{n^2}-(\lambda-A(t))^nK_{\lambda,n}(t,s),
				\end{eqnarray*}
				where $K_{\lambda,n}(t,s)=\sum_{i=1}^nC_n^i
				s^{n(n-i)}(\lambda-A(t))^{n(i-1)}[L_\lambda(t,s)]^{ni}.$
				Hence, we obtain
				$$[D_\lambda(t,s)]^n[D_{\lambda,n}(t,s)]^n +(\lambda-A(t))^nK_{\lambda,n}(t,s) =s^{n^2}.$$
				Finally, the commutativity is clear.
			\end{enumerate}
		\end{proof}

	We start by this result.
	
	\begin{prop}\label{p1} Let $A(t) $ be a closed  and densely defined  generator of  a $C_{0}-$quasi-semigroup $(R(t,s))_{t,s\geq 0}$ on  a 	Banach  space $X$. If \; $Rg[e^{\lambda s}-R(t,s)]^{n}$ is closed, then $Rg[\lambda-A(t)]^{n}$ is also closed.
	\end{prop}

	\begin{proof}
		Let $(y_n)_{n\in\mathbb{N}}\subseteq X$ such that $y_n\rightarrow y\in X$ and there exists $(x_n)_{n\in\mathbb{N}}\subseteq \mathcal{D}$ satisfying
		$$(\lambda-A(t))^{n}x_n=y_n.$$
		By theorem \ref{t2}, for all $n\in\mathbb{N}^*,$ there exists $D_{\lambda,n}(t,s)$, $K_{\lambda,n}(t,s)\in \mathcal{B}(X)$ such that,
		$$(\lambda-A(t))^nK_{\lambda,n}(t,s)+[D_{\lambda,n}(t,s)]^n[D_\lambda(t,s)]^n =s^{n^2}.$$
		Hence, we conclude that
		\begin{eqnarray*}
			[e^{\lambda s}-R(t,s)]^{n}[D_{\lambda,n}(t,s)]^nx_n
			&=& D_\lambda(t,s)^{n}(\lambda - A(t))^{n}[D_{\lambda,n}(t,s)]^nx_n \; , \; (by \;  theorem \;  \ref{t1} )\\
			&=& [D_{\lambda,n}(t,s)]^nD_\lambda(t,s)^{n}(\lambda -A(t))^{n}x_n\\
			&=& [D_{\lambda,n}(t,s)]^nD_\lambda(t,s)^{n}y_n\\
			&=& s^{n^2}y_n-(\lambda-A(t))^{n}K_{\lambda,n}(t,s)y_n.
		\end{eqnarray*}
		Thus, $$s^{n^2}y_n-(\lambda-A(t))^{n}K_{\lambda,n}(t,s)y_n \in Rg[e^{\lambda s}-R(t,s)].$$
		Therefore, since $Rg[e^{\lambda s}-R(t,s)]^n$ is closed, $K_{\lambda,n}(t,s)$ is bounded linear and
		$s^{n^2}y_n-(\lambda-A(t))^{n}K_{\lambda,n}(t,s)y_n$ converges to  $s^{n^2}y-(\lambda-A(t))^{n}K_{\lambda,n}(t,s)y,$ we conclude that
		$$s^{n^2}y-(\lambda-A(t))^{n}K_{\lambda,n}(t,s)y\in Rg[e^{\lambda s}-R(t,s)].$$
		Then there exists $z\in X$ such that
		$$[e^{\lambda s}-R(t,s)]^{n}z=s^{n^2}y-(\lambda-A(t))^{n}K_{\lambda,n}(t,s)y.$$
		Hence for all $s\neq 0$, we have 
		\begin{eqnarray*}
			y &=& \frac{1}{s^{n^2}}[[e^{\lambda s}-R(t,s)]^{n}z+(\lambda-A(t))L_\lambda(t,s)y];\\
			&=& \frac{1}{s^{n^2}}[(\lambda-A(t))^{n}D_\lambda(t,s)^{n}z+(\lambda-A(t))^{n}K_{\lambda,n}(t,s)y];\\
			&=& \frac{1}{s^{n^2}}(\lambda-A(t))^{n}[D_\lambda(t,s)^{n}z+K_{\lambda,n}(t,s)y].\\
		\end{eqnarray*}
		Finally, we obtain $$y\in Rg(\lambda-A(t))^{n}.$$
	\end{proof}
	
	The following result discusses the semi-Fredholm spectrum.

	
	\begin{thm}\label{t3} Let $A(t) $ be a closed  and densely defined  generator of  a $C_{0}-$quasi-semigroup  $(R(t,s))_{t,s\geq 0}$ on  a 	Banach  space $X$. For all $\lambda\in\mathbb{C}$ and all $t,s\geq 0$, we have
		\begin{enumerate}
			\item  $e^{\sigma_{e_+}(A(t))s} \subseteq \sigma_{e_+}(R(t,s));$
			\item  $e^{\sigma_{e_-}(A(t))s} \subseteq \sigma_{e_-}(R(t,s));$
			\item $ e^{\sigma_{e_\pm(A(t))s}} \subseteq  \sigma_{e_\pm}(R(t,s)).$
		\end{enumerate}
	\end{thm}

	\begin{proof}
		\begin{enumerate}
			\item Suppose that $e^{\lambda s}  \notin  \sigma_{e_+}(R(t,s))$, then there exists $n \in \mathbb{N}$ such that
			$\alpha[e^{\lambda s} -R(t,s)]=n$ and
			$Rg[e^{\lambda s} -R(t,s)]$ is closed.\\
			By theorem \ref{t1}, we obtain $$N(\lambda -A(t) )\subset N[e^{\lambda s} -R(t,s)],$$
			then $$\alpha(\lambda-A(t)) \leq n.$$
			On the other hand, from Proposition \ref{p1}, we deduce that $Rg(\lambda -A(t))$ is closed.\\
			Therefore $\lambda- A(t) \in \Phi_+(\mathcal{D})$,
			Then$\lambda\notin \sigma_{e_+}(A(t)).$\\
			
			\item Suppose that $e^{\lambda s}  \notin  \sigma_{e_-}(R(t,s))$, then there exist $n\in\mathbb{N}$ such that
			$\beta[e^{\lambda s} -R(t,s)]=n $.\\
			By theorem \ref{t1}, we obtain $$ Rg[e^{\lambda s} - R(t,s)]\subseteq Rg(\lambda-A(t)),$$
			then $\beta(\lambda-A(t))\leq n$ and hence,  $\lambda\notin\sigma_{e_-}(A(t))$
			\item It is automatic by the previous assertions of this theorem.
		\end{enumerate}
	\end{proof}

	
	\begin{prop}\label{p2}Let $A(t)$ be the generator of a $C_0$-quasi-semigroup $(R(t,s))_{t,s\geq 0}$. For all $\lambda\in\mathbb{C}$ and all $t,s\geq 0$,we have
		\begin{enumerate}
			\item If $d[e^{\lambda s}-R(t,s)]=n,$ then  $d[\lambda-A(t)]\leq n.$
			\item If $a[e^{\lambda s}-R(t,s)]=n,$ then  $a[\lambda-A(t)]\leq n.$
		\end{enumerate}
	\end{prop}

	\begin{proof}
		$\,$\\
		\begin{enumerate}
			\item Let $y\in Rg[\lambda-A(t)]^n$, then there exists $x\in \mathcal{D}^n$ (domain of $A(t)^n$) satisfying, $$(\lambda-A(t))^nx=y.$$
			Since $d[e^{\lambda s}-R(t,s)]=n,$ therefore
			$Rg[e^{\lambda s}-R(t,s)]^n=Rg[e^{\lambda s}-R(t,s)]^{n+1}.$
			Hence, there exists $z\in X$ such that
			$$ [e^{\lambda s}-R(t,s)]^nx=[e^{\lambda s}-R(t,s)]^{n+1}z.$$
			On the other hand, by theorem \ref{t2}, we have, $$(\lambda-A(t))^nK_{\lambda,n}(t,s)+[D_{\lambda,n}(t,s)]^n[D_\lambda(t,s)]^n=s^{n^2},$$
			Thus we have,
			\begin{eqnarray*}
				s^{n^2}y&=&(\lambda-A(t))^ns^{n^2}x\\
				&=& (\lambda-A(t))^n[(\lambda-A(t))^nK_{\lambda,n}(t,s)+[D_{\lambda,n}(t,s)]^n[D_\lambda(t,s)]^n]x\\
				&=&(\lambda-A(t))^n(\lambda-A(t))^nK_{\lambda,n}(t,s)x +[D_{\lambda,n}(t,s)]^n
				(\lambda-A(t))^n[D_\lambda(t,s)^n]x\\
				&=&(\lambda-A(t))^{2n}K_{\lambda,n}(t,s)x+[D_{\lambda,n}(t,s)]^n[e^{\lambda s}-R(t,s)]^nx\\
				&=&(\lambda-A(t))^{2n}K_{\lambda,n}(t,s)x+[D_{\lambda,n}(t,s)]^n[[e^{\lambda s}-R(t,s)]^{n+1}z]\\
				&=&(\lambda-A(t))^{2n}K_{\lambda,n}(t,s)x+[D_{\lambda,n}(t,s)]^n[(\lambda-A(t))^{n+1}
				[D_\lambda(t,s)]^{n+1}z]\\
				&=&(\lambda-A(t))^{n+1}[(\lambda-A(t))^{n-1}K_{\lambda,n}(t,s)x+
				[D_{\lambda,n}(t,s)]^n[D_\lambda(t,s)]^{n+1}z].
			\end{eqnarray*}
			Therefore, we conclude that $y\in Rg[\lambda-A(t)]^{n+1}$ and hence,  $$Rg[\lambda-A(t)]^n=Rg[\lambda-A(t)]^{n+1}.$$
			Finally, we conclude that $$d(\lambda-A(t))\leq n.$$
			\item Let $x\in N(\lambda-A(t))^{n+1}$ and we suppose that $a[e^{\lambda s}-R(t,s)]=n $, then we obtain
			$$N[e^{\lambda s}-R(t,s)]^n=
			N[e^{\lambda s}-R(t,s)]^{n+1}.$$
			From Corollary \ref{c1}, we have
			$$N(\lambda-A(t))^{n+1}\subseteq N[e^{\lambda s}-R(t,s)]^{n+1},$$
			hence $$x\in N[e^{\lambda s}-R(t,s)]^n.$$
			Thus we have,
			\begin{eqnarray*}
				s^{n^2}(\lambda-A(t))^nx&=&(\lambda-A(t))^n[(\lambda-A(t))^nK_{\lambda,n}(t,s)+[D_{\lambda,n}(t,s)]^n[D_\lambda(t,s)]^n]x\\
				&=&(\lambda-A(t))^{2n}K_{\lambda,n}(t,s)x+
				[D_{\lambda,n}(t,s)]^n(\lambda-A(t))^n[D_\lambda(t,s)]^nx\\
				&=&K_{\lambda,n}(t,s)(\lambda-A(t))^{n-1}(\lambda-A(t))^{n+1}x+[D_{\lambda,n}(t,s)]^n
				[e^{\lambda s}-R(t,s)]^nx\\
				&=&K_{\lambda,n}(t,s)(\lambda-A(t))^{n-1}(\lambda-A(t))^{n+1}x+[D_{\lambda,n}(t,s)]^n
				[e^{\lambda s}-R(t,s)]^nx\\
				&=& 0.
			\end{eqnarray*}
			Therefore, we obtain $x\in N(\lambda-A(t))^n$ and hence $$a(\lambda-A(t))\leq n.$$
		\end{enumerate}
	\end{proof}

\begin{coro} Let $A(t) $ be a closed  and densely defined  generator of  a $C_{0}-$quasi-semigroup  $(R(t,s))_{t,s\geq 0}$ on  a 	Banach  space $X$. For all $\lambda\in\mathbb{C}$ and all $t,s\geq 0$, we have
	\begin{enumerate}
		\item $e^{\sigma_{a}(A(t))s}\subseteq \sigma_{a}(R(t,s));$
		\item $e^{\sigma_{d}(A(t))s} \subseteq \sigma_{d}(R(t,s)).$
		\item $e^{\sigma_{D}(A(t))s} \subseteq \sigma_{D}(R(t,s)).$
	\end{enumerate}
\end{coro}

\begin{proof}
	Immediately comes from propositions \ref{p2} and  \ref{p1} .
\end{proof}


	The following theorem examines the semi-Browder spectrum.
	
	
	\begin{thm}Let $A(t) $ be a closed  and densely defined  generator of  a $C_{0}-$quasi-semigroup  $(R(t,s))_{t,s\geq 0}$ on  a 	Banach  space $X$. For all $\lambda\in\mathbb{C}$ and all $t,s\geq 0$, we have
		\begin{enumerate}
			\item $e^{\sigma_{{Br}_+}(A(t))s} \subseteq \sigma_{{Br}_+}(R(t,s));$
			\item $e^{\sigma_{{Br}_-}(A(t))s}\subseteq \sigma_{{Br}_-}(R(t,s));$
			\item $e^{\sigma_{{Br}_\pm}(A(t))s}\subseteq \sigma_{{Br}_\pm}(R(t,s)).$
		\end{enumerate}
	\end{thm}

	\begin{proof}
		\begin{enumerate}
			\item Suppose that $e^{\lambda s}\notin\sigma_{Br_+}(R(t,s))$, then there exist $n,m\in\mathbb{N}$ such that
			$\alpha[e^{\lambda s}-R(t,s)]=m $,
			$Rg[e^{\lambda s}-R(t,s)]$ is closed and $a[e^{\lambda s} - R(t,s)]=n.$
			From Corollary \ref{c1} and Propositions \ref{p1} and \ref{p2}, we obtain
		\begin{center}
				$\alpha(\lambda-A(t))\leq m$, $Rg(\lambda-A(t))$ is closed and $a(\lambda-A(t))\leq n.$
		\end{center}
			Therefore $\lambda- A(t) \in \Phi_+(\mathcal{D})$ and $a(\lambda-A(t))< +\infty$ and hence, $\lambda\notin \sigma_{Br_+}(A).$
			
			\item Suppose that $e^{\lambda s} \notin\sigma_{Br_-}(R(t,s))$, then there exist $n,m\in\mathbb{N}$ such that
			$\beta[e^{\lambda s}-R(t,s)]=m$ and $d[e^{\lambda s}-R(t,s)]=n.$
			By corollary \ref{c1} and Proposition \ref{p2}, we obtain
			\begin{center}
				$\beta(\lambda-A(t))\leq m$ and $d(\lambda-A(t))\leq n.$
			\end{center}
			Therefore $\lambda-A(t) \in \Phi_-(\mathcal{D})$ and $d(\lambda-A(t))< +\infty$ and hence, $\lambda\notin\sigma_{Br_-}(A(t)).$
			\item It is automatic by the previous assertions of this theorem.
		\end{enumerate}
	\end{proof}

	\begin{prop}\label{p6} Let $A(t) $ be a closed  and densely defined  generator of  a $C_{0}-$quasi-semigroup  $(R(t,s))_{t,s\geq 0}$ on  a 	Banach  space $X$. For all $\lambda\in\mathbb{C}$ and all $t,s\geq 0$, we have
		\begin{enumerate}
			\item If $d_e[e^{\lambda s}-R(t,s)]=n,$ then  $d_e[\lambda - A(t)]\leq n;$
			\item If $a_e[e^{\lambda s}-R(t,s)]=n,$ then  $a_e[\lambda - A(t)]\leq n.$
		\end{enumerate}
	\end{prop}
	
	\begin{proof}
		\begin{enumerate}
			\item Suppose that $d_e[e^{\lambda s}-R(t,s)]=n$, Since $Rg[e^{\lambda s}-R(t,s)]^n\subseteq Rg(\lambda-A(t))^n$
			we define the linear surjective application $\phi$ by
			\begin{eqnarray*}
				\phi: Rg(\lambda-A(t))^n &\rightarrow& Rg[e^{\lambda s}-R(t,s)]^n/ Rg[e^{\lambda s}-R(t,s)]^{n+1},\\
				\; y=(\lambda-A(t))^nx &\mapsto & [e^{\lambda s}-R(t,s)]^nx +Rg[e^{\lambda s}-R(t,s)]^{n+1}.
			\end{eqnarray*}
			Thus, by isomorphism Theorem, we obtain
			$$Rg(\lambda-A(t))^n/ N(\phi)\simeq Rg[e^{\lambda s}-R(t,s)]^n/ Rg[e^{\lambda s}-R(t,s)]^{n+1}.$$
			Therefore $$dim(Rg(\lambda-A(t))^n/ N(\phi))= d_e[e^{\lambda s}-R(t,s)]=n.$$
			And since
			$N(\phi) \subseteq  Rg[e^{\lambda s}-R(t,s)]^{n+1}\subseteq Rg(\lambda-A(t))^{n+1},$
			then $$Rg(\lambda-A(t))^n/ Rg(\lambda-A(t))^{n+1}\subseteq Rg(\lambda-A(t))^n/ N(\phi).$$
			Finally, we obtain
			$$d_e(\lambda-A(t))=dim(Rg(\lambda-A(t))^n/ R(\lambda-A(t))^{n+1} )\leq dim(Rg(\lambda-A(t))^n/ N(\phi))= n.$$
			\item Suppose that $$a_e[e^{\lambda s}-R(t,s)]=n.$$
		And since $N(\lambda-A)^{n+1}\subseteq N[e^{\lambda s}-R(t,s)]^{n+1},$
			we define the linear application $\psi$ by
			\begin{eqnarray*}
				\psi: N(\lambda-A(t))^{n+1} &\rightarrow& N[e^{\lambda s}-R(t,s)]^{n+1}/ N[e^{\lambda s}-R(t,s)]^n,\\
				x &\mapsto & x + N[e^{\lambda s}-R(t,s)]^n.
			\end{eqnarray*}
			Thus, by isomorphism Theorem, we obtain
			$$N(\lambda-A(t))^{n+1}/ N(\psi)\simeq Rg(\psi)\subseteq N[e^{\lambda s}-R(t,s)]^{n+1}/ N[e^{\lambda s}-R(t,s)]^n.$$
			Therefore $$dim (N(\lambda-A)^{n+1}/ N(\psi))\leq a_e[e^{\lambda s}-R(t,s)]=n.$$
			And since
			$N(\psi) \subseteq  N[e^{\lambda s}-R(t,s)]^n\subseteq Rg(\lambda-A(t))^n,$
			then  $$N(\lambda-A(t))^{n+1}/ N(\lambda-A(t))^n\subseteq N(\lambda-A(t))^{n+1}/ N(\psi).$$
			Finally, we obtain $$a_e(\lambda-A)=dim(N(\lambda-A(t))^{n+1}/ N(\lambda-A(t))^n)\leq dim (N(\lambda-A(t))^{n+1}/ N(\psi))\leq n .$$
		\end{enumerate}
	\end{proof}
	

	We will discuss in the following result the essential ascent and descent spectrum.
	

	\begin{thm} Let $A(t) $ be a closed  and densely defined  generator of  a $C_{0}-$quasi-semigroup  $(R(t,s))_{t,s\geq 0}$ on  a 	Banach  space $X$. For all $\lambda\in\mathbb{C}$ and all $t,s\geq 0$, we have
		\begin{enumerate}
			\item $e^{\sigma_{a_e}(A(t))s}\subseteq \sigma_{a_e}(R(t,s));$
			\item $e^{\sigma_{d_e}(A(t))s} \subseteq \sigma_{d_e}(R(t,s)).$
		\end{enumerate}
	\end{thm}
	
	\begin{proof}
		\begin{enumerate}
			\item Suppose that , $ e^{\lambda s} \notin \sigma_{a_e}(R(t,s)).$
			Then there exists $n\in\mathbb{N}$ satisfying $$a_e[e^{\lambda s}-R(t,s)]=n.$$
			Therefore, by Proposition \ref{p6}, we obtain
			$a_e[\lambda-A(t)]\leq n$ and hence $$\lambda\notin\sigma_{a_e}(A(t)).$$
			\item Suppose that $$e^{\lambda s}\notin \sigma_{d_e}(R(t,s)).$$
			Then there exists $n\in\mathbb{N}$ satisfying $$d_e[e^{\lambda s}-R(t,s)]=n.$$
			Therefore, by Proposition \ref{p6}, we obtain
			$d_e[\lambda-A(t)]\leq n$ and hence $\lambda\notin\sigma_{d_e}(A(t)).$
		\end{enumerate}
	\end{proof}

	{
		}
	
\end{document}